\documentclass[11pt]{amsart}
\usepackage[active]{srcltx}
\usepackage{amsmath, amsfonts,amsthm,times,graphics,color}
\newcommand{\vertiii}[1]{{\left\vert\kern-0.25ex\left\vert\kern-0.25ex\left\vert #1
    \right\vert\kern-0.25ex\right\vert\kern-0.25ex\right\vert}}
 \makeatletter
\renewcommand*\subjclass[2][2000]{%
  \def\@subjclass{#2}%
  \@ifundefined{subjclassname@#1}{%
    \ClassWarning{\@classname}{Unknown edition (#1) of Mathematics
      Subject Classification; using '1991'.}%
  }{%
    \@xp\let\@xp\subjclassname\csname subjclassname@#1\endcsname
  }%
}
 \makeatother
\usepackage{enumerate,amssymb,  mathrsfs,yhmath}

\newtheorem{theorem}{Theorem}[section]
\newtheorem{lemma}[theorem]{Lemma}
\newtheorem*{lemma*}{Lemma}

\newtheorem{corollary}[theorem]{Corollary}

\def\1ton{1,2,\ldots,n}

\usepackage{amssymb}
\usepackage{amsthm}
\usepackage{mathrsfs, amsfonts, amsmath}
\usepackage{graphicx}
\usepackage[margin=3.3cm]{geometry}

\theoremstyle{definition}

\theoremstyle{remark}

\numberwithin{equation}{section}

\renewcommand{\imath}{i} 

\def\XXint#1#2#3{{\setbox0=\hbox{$#1{#2#3}{\int}$}
\vcenter{\hbox{$#2#3$}}\kern-.5\wd0}}

\def\ge{\geqslant}
\begin{document}
\title{Zygmund theorem for harmonic quasiregular mappings }  \subjclass[2020]{Primary 30H10 }


\author{David Kalaj}
\address{University of Montenegro, Faculty of Natural Sciences and
Mathematics, Cetinjski put b.b. 81000 Podgorica, Montenegro}
\email{davidk@ucg.ac.me}


\begin{abstract}
Let $K\ge 1$.
We prove Zygmund theorem for $K-$quasiregular harmonic mappings in the unit disk $\mathbb{D}$ in the complex  plane by providing a constant $C(K)$ in the inequality $$\|f\|_{1}\le C(K)(1+\|\mathrm{Re}\,(f)\log^+ |\mathrm{Re}\, f|\|_1),$$ provided that $\mathrm{Im}\,f(0)=0$. Moreover for a quasiregular harmonic mapping $f=(f_1,\dots, f_n)$ defined in the unit ball $\mathbb{B}\subset \mathbb{R}^n$, we prove the asymptotically sharp inequality
$$\|f\|_{1}-|f(0)|\le (n-1)K^2(\|f_1\log f_1\|_1- f_1(0)\log f_1(0)),$$ when $K\to 1$,  provided that $f_1$ is positive.

\end{abstract}
\maketitle
\tableofcontents
\section{Introduction}
In this paper $\mathbb{D}$ is the unit disk in the complex plane $\mathbb{C}$ and $\mathbb{B}$ is the unit ball in the Euclidean space $\mathbb{R}^n$. By $\mathbb{S}$ we denote the unit sphere.
\subsection{Hardy class}
Let $p>0$ and  $$\mathbf{h}^p(\mathbb{B})=\{f: \mathbb{B}\to \mathbb{R}^n, \text{so that } \|f\|_p=\sup_{0<r<1}M_p(r,f)<\infty\},$$ where  $$M_p(r,f)=\left(\int_{\mathbb{S}}|f(r\zeta)|^pd\sigma(\zeta)\right)^{1/p}.$$ Here $\sigma$ is the standard normalized area measure of the unit sphere $\mathbb{S}$. Now define for $x>0$, $\log^+(x):=\max\{\log x, 0\}$. Then define the space $\mathbf{h}\log^+\mathbf{h}$ of functions defined in the unit ball so that $$\sup_{0<r<1} \int_{\mathbb{S}}|f(r\zeta)|\log^+|f(r\zeta)|d\sigma(\zeta)<\infty.$$
\subsection{Motivation}
We formulate the following  Zygmund theorem \cite[p.~254]{Zyg}, which is the motivation for our study.
Assume that $f=u+iv$ is a holomorphic function defined in the unit disk $\mathbb{D}$ such that $v(0)=0$. Assume further that $|u|\log^+|u|\in \mathbf{h}^1(\mathbb{D}) $. We write it shortly by $u\in \mathbf{h}\log^+\mathbf{h}$. Then  $f\in \mathbf{h}^1$ and we have the inequality
\begin{equation}\label{aaa}
\|f\|_1\le A(\||u|\log^+|u|\|_1+1),
\end{equation}
where $A$ is an absolute constant. It is well-known that the condition $u\in \mathbf{h}^1(\mathbb{D})$ does not imply that $f\in \mathbf{h}^1(\mathbb{D})$, so \eqref{aaa} is in some sense the best possible inequality \cite[p.~60]{duren}. On the other hand for $p>1$, $u\in \mathbf{h}^p(\mathbb{D})$  imply that $f\in \mathbf{h}^p(\mathbb{D})$ and moreover there is a constant $c_p$ so that $\|f\|_{p}\le c_p\|u\|_{p}$ whenever $\mathrm{Im} f(0)=0$. This is a well-known M. Riesz theorem. For some recent  extensions of M. Riesz theorem we refer to the papers \cite{tams, aimzhu, mel, melmar, chenarxiv, kalajarxiv2, kalajarxiv3, studia}.

The sharpness of constant $A$ in Zygmund theorem as well as its connection with M. Riesz theorem  is discussed by Pichorides in \cite{Pi}.

This paper aims  to extend Zygmund theorem to the class of quasiregular harmonic mappings in the unit disk and in the unit ball.
\subsection{Quasiregular and harmonic mappings}
A continuous and nonconstant mapping $f : G\to \Bbb R^n$, $n\ge
2$, in the local Sobolev space $W^{1,n}_{loc} (G,\Bbb R^n)$ is
$K$-quasiregular, $K\ge 1$, if $$|Df(x)| \le K \ell(f' (x))$$ for almost
every $x\in G$, where $G$ is an open subset of $\mathbb{R}^n$. Here $Df(x)$ is the formal differential matrix and $$|Df(x)|=\sup_{|h|=1}|f'(x)h|,\ \  \ell(Df(x))=\inf_{|h|=1}|f'(x)h|.$$  Here and in the sequel $|\zeta|$ is the Euclidean norm of a vector $\zeta\in \mathbb{R}^n$. Further we denote by  $\left<\zeta,\xi\right>$ the inner product of vectors $\zeta,\xi \in  \mathbb{R}^n$.




A smooth mapping $w:G\to \mathbb{R}^n$ is called harmonic if it satisfies the Laplace equation $\Delta u=0$.
The solution of the equation $\Delta w=g$ (in the sense of
distributions see \cite{Her}) in the  ball $B_R=R\cdot \mathbb{B}$, satisfying the
boundary condition $w|_{S_R}=f\in L^1(S_R)$, where $S_R=R\cdot\mathbb{S}$ is given by
\begin{equation}\label{green}w(x)=\int_{S_R}P(x,\eta)f(\eta)d\sigma(\eta)-\int_{B_R}G(x,y)g(y)dV(y),\,
|x|<1.\end{equation} Here
\begin{equation}\label{poisson}P(x,\eta)=\frac{R^2-|x|^2}{R|x-\eta|^n}\end{equation}
is the Poisson kernel and $d\sigma$ is the  surface $n-1$ dimensional
measure of the Euclidean sphere which satisfies the condition:
$\int_{\mathbb{S}}P(x,\eta)d\sigma(\eta)\equiv 1$. The first integral
in (\ref{green}) is called the Poisson integral and is usually
denoted by $u(x)=P[f](x)$. Then $u$ is a  harmonic mapping. The function
\begin{equation}\label{green1}G(x,y)=\left\{
                                       \begin{array}{ll}
                                         -\frac{1}{2\pi}\log \frac{R|x-y|}{|R^2-x\overline{y}|}, & \hbox{for $n=2$;} \\
                                         c_n\left(\frac{1}{|x-y|^{n-2}}
-\frac{1}{(R^2+|x|^2|y|^2/R^2-2\left<x,y\right>)^{(n-2)/2}}\right), & \hbox{for $n\ge 3$,}
                                       \end{array}
                                     \right.\end{equation}
where
\begin{equation}\label{cen}c_n=\frac{1}{(n-2)\omega_{n-1}}\end{equation} and
$\omega_{n-1}$ is the measure of $\mathbb{S}$, is the Green function of
the  ball $B_R$.

If $n=2$, then we use simultaneously the notation of $k-$quasiregular mappings for $k=(K-1)/(K+1)$.
If $f:\mathbb{D}\to \mathbb{C}$ is harmonic mapping, then there exist two holomorphic functions $g$ and $h$ ($h(0)=0$), in the unit disk, so that $f(z) = g(z) + \overline{h(z)}$. Then $f$ is sense-preserving and $k-$quasiregular, if and only if $|h'(z)|\le k|g'(z)|$, $z\in \mathbb{D}$.
\subsection{Main results}
The following three theorems are main results.

\begin{theorem}\label{th1}
Let $f=g+\bar h$ be a $K-$quasiregular mapping defined in the unit disk $\mathbb{D}$ such that $\mathrm{Im}\, f(0)=0$. Assume further that $u=\mathrm{Re}\, f\in \mathbf{h}\log^+ \mathbf{h}$. Then $f\in \mathbf{h}^1$ and we have the inequality
\begin{equation}\label{zygg1}\int_0^{2\pi}|f(r e^{it})|\frac{dt}{2\pi}\le C_0(K)\left(1+\int_0^{2\pi} |u(re^{it})| \log^+ |u(re^{it})| \frac{dt}{2\pi}\right)\end{equation} for every $r\in (0,1)$ and in particular
\begin{equation}\label{zygg2}\|f\|_1    \le C_0(K)\left(\|u\log^+ u\|_{1}+1\right).\end{equation}
\end{theorem}
For a quasiregular mapping with a positive real part, we have the following asymptotically sharp result.
\begin{theorem}\label{zyg}
Assume that $f=u+ iv$ is a harmonic $K-$quasiregular function defined in the unit disk $\mathbb{D}$ so that $v(0)=0$ and assume further that $u(z)>0$ for $z\in\mathbb{D}$. If $u\in \mathbf{h}\log^+ \mathbf{h}$, then $f\in \mathbf{h}^1$ and we have the inequality

 \begin{equation}\label{zygmund}\int_0^{2\pi}|f(r e^{it})|\frac{dt}{2\pi}\le K^2 \left(e^{-1+K^{-2}}+\int_0^{2\pi} u(re^{it}) \log u(re^{it}) \frac{dt}{2\pi}\right)\end{equation} for every $r\in (0,1)$ and in particular
\begin{equation}\label{zygmund1}\|f\|_1\le K^2 \left(e^{-1+K^{-2}}+\int_{\mathbb{T}} u \log u \frac{|dz|}{2\pi}\right)\end{equation}
and  $$\|v\|_1\le \|f\|_1  \le K^2\left(\|u\log^+ u\|_{1}+1\right).$$  The constants are asymptotically sharp when $K\to 1$.
\end{theorem}
A straightforward consequence of Theorem~\ref{zyg} is the following corollary which is interesting in its own right.
\begin{corollary}
If $f=u+i v$ is an analytic function of the unit disk into the vertical strip $(0,1)\times \mathbb{R}$ so that $v(0)=0$, then $$\|f\|_{H^1}<1.$$ The constant $1$ is sharp as can be seen by the following example $f(z) = \frac{n}{(n+1)},$ when $n\to \infty$. Indeed, we can use any admissible function of $f$, with $f(0)=n/(n+1)$.
\end{corollary}

For higher-dimensional settings we prove
\begin{theorem}\label{zyg44}
Assume that $f=(f_1,\dots, f_n):\mathbb{B}\to \mathbb{R}^n$ is a harmonic $K-$quasiregular function defined in the unit ball $\mathbb{B}$ so that $u(x)=f_1(x)>0$ for $z\in\mathbb{B}$. If $u\in \mathbf{h}\log^+ \mathbf{h}$, then $f\in \mathbf{h}^1$ and we have the inequality

\begin{equation}\label{zygmund99}\|f\|_1-|f(0)|\le K^2(n-1)\left(\int_{\mathbb{S}}\left(u(\zeta)\log u(\zeta)-u(0)\log u (0)\right) d\sigma(\zeta)\right)\end{equation}
and if $|f(0)|=u(0)$, then  \begin{equation}\label{zyg99}\|f\|_1\le K^2(n-1)( e^{-1+K^{-2}/(n-1)}+\|u\log^+u \|_1).\end{equation}
The constant $K^2(n-1)$ in \eqref{zygmund99} is asymptotically sharp as $K\to 1$.
\end{theorem}

\section{Proof of main results}

\begin{proof}[Proof of Theorem~\ref{th1}]

We lose no generality in assuming that $r=1$. Assume first that $\mathrm{Re}(f(0))=0$. Since $f(0)=0,$ we can assume that $h(0)=g(0)=0$.
 Note first that if $H$ is an analytic function defined in the unit disk with $H(0)=0$, then we have the  inequalities \begin{equation}\label{monday}\|H\|_p\le c_1(p)\|G[H]\|_p\text{\ \ and \ \ }\|G[H]\|_p\le c_2(p)\|H\|_p,\end{equation} where \begin{equation}\label{cand}G[H](z)=\left(\int_{0}^1|H'(\rho z)|^2(1-\rho)d\rho\right)^{1/2}.\end{equation} This hold due to Calderon type theorem of Pavlovi\' c \cite[Theorem~10.6]{pavlovic}. We apply this inequality for $p=1$ and $H\in\{g,h,g+h\}$. Let $u = \mathrm{Re}\,(g+h)$, then $\tilde u=\mathrm{Im}\,(g+h)\in \mathbf{h}^1(\mathbb{D})$ and from the classical Zygmund theorem \cite[p.~57]{duren}, we have $$\|\tilde u\|_{H^1}\le A\||u|\log^+|u|\|_{L^1}+B,$$ where $$ A=1 , \ \  B=6\pi e .$$ 
 Since $|u|\le 1+|u|\log^+|u|$, we get the following $$\|g+h\|_{H^1}\le \|u\|_1+\|\tilde u\|_1\le  2\||u|\log^+|u|\|_{L^1}+6\pi e+1. $$
Since $f$ is quasiregular we have \begin{equation}\label{quas}
|g'(z)|\le \frac{|g'(z)+h'(z)|}{1-k}, \ \ \ |h'(z)|\le \frac{k|g'(z)+h'(z)|}{1-k},
\end{equation}
where $k=(K-1)/(K+1)$.
 Denote the constants in \eqref{monday} by $c_1=c_1(1)$ and $c_2=c_2(1)$. Then by \eqref{monday}  and \eqref{quas}, we have $$\|g\|_{H^1}\le c_1 \|G(g)\|_1\le  \frac{c_1}{1-k}\|G(g+h)\|_1\le  \frac{c_1\cdot c_2}{1-k}\|g+h\|_{H^1}$$  and

 $$\|h\|_{H^1}\le c_1 \|G(h)\|_1\le  \frac{c_1 k}{1-k}\|G(g+h)\|_1\le  \frac{c_1\cdot c_2 k }{1-k}\|g+h\|_{H^1}.$$
 Thus $$\|f\|_1\le  (\|g\|_{H^1}+\|h\|_{H^1})\le CK(2\||u|\log^+|u|\|_{L^1}+6\pi e+1), $$ where  $C=c_1\cdot c_2$.

  The last inequality follows from the fact that $|f|$ is a subharmonic function.

 Thus for $C(K)=2(6\pi e+1) c_1 c_2 K$ we proved \eqref{zygg1} and \eqref{zygg2} for the case $u(0)=0$.

 Assume now that $u(0)\neq 0$. Let $F(z) = f(z) - u(0)$. Then by applying the previous case we have \begin{equation}\label{g11}\int_0^{2\pi}|f(r e^{it})-u(0)|\frac{dt}{2\pi}\le C(K)\left(1+\int_0^{2\pi} |u(re^{it})-u(0)| \log^+ |u(re^{it})-u(0)| \frac{dt}{2\pi}\right).\end{equation}

 Now we  use the following convex function $k(t) = t\log^{+}(t)$ for $t>0$, in order to conclude that $g(z)=k(|u(z)|)=|u(z)|\log^{+}(|u(z)|)$ is subharmonic. Then we get the inequality \begin{equation}\label{g00} g(0)=|u(0)|\log^{+}(|u(0)|)\le \frac{1}{2\pi}\int_{\mathbb{T}} |u(z)|\log^{+}(|u(z)|) |dz|.\end{equation} Further, if $|u(0)|>1$, then
 $$ g(0)=|u(0)|\log _{+}(|u(0)|)=|u(0)|\log (|u(0)|)\le  \frac{1}{2\pi}\int_{\mathbb{T}} |u(z)|\log^{+}(|u(z)|) |dz|,$$ and therefore $$|u(0)|\le e+ \frac{1}{2\pi}\int_{\mathbb{T}} |u(z)|\log^{+}(|u(z)|) |dz|.$$

 Then from \eqref{g00}, \eqref{g11} and Lemma~\ref{froml} below we have
 \[\begin{split}
 \|f\|_{1}&\le |u(0)|+\|f-u(0)\|_{1}
  \\ &\le e+ \frac{1}{2\pi}\int_{\mathbb{T}} |u(z)|\log^{+}(|u(z)|) |dz|
  \\& +C(K)\left(1+\int_0^{2\pi} |u(re^{it})-u(0)| \log^+ |u(re^{it})-u(0)| \frac{dt}{2\pi}\right)
  \\ &\le   e+ \frac{1}{2\pi}\int_{\mathbb{T}} |u(z)|\log^{+}(|u(z)|) |dz|
 \\&+C(K)\left(1+4/e + 2|u(0)| \log^+ |u(0)|+  \frac{2}{2\pi}\int_{\mathbb{T}} |u(z)|\log^{+}(|u(z)|) |dz|\right)
 \\&\le (1+4 C(K))\left(1+\frac{1}{2\pi}\int_{\mathbb{T}} |u(z)|\log^{+}(|u(z)|) |dz|\right).\end{split}\] In the last inequality, we applied the bound
\[
\left(c+\frac{4 c}{e}+e\right)+(1+4 c)X <(1+4c)(1+X),
\]
which holds trivially under the assumption that \( c=C(K) > 2 \).  This completes the proof of Theorem~\ref{th1}, modulo the following lemma.
\end{proof}
 \begin{lemma}\label{froml}
 For every real $a$ and $x$ we have \[
f(a,x):=|a+x| \max(0,\log |a+x|)
- 2 \left(  |a| \max(0,\log |a|) + |x| \max(0,\log |x|) \right) \leq 4/e.
\]
  \end{lemma}

\begin{proof}
Assume without loosing the generality that $a+x\ge 0$ and $x\ge 0$. Then  \[
f(a, x) =
\begin{cases}
(a + x)  \log(a + x) - 2  \left( a  \log(a) + x  \log(x)\right), & \text{if } a > 0, \, a + x > 1, \, x > 1, \\
(a + x)  \log(a + x) - 2  \left( a  \log(a)\right), & \text{if } a > 0, \, a + x > 1, \, x \leq 1, \\
-2  \left( a  \log(a) + x  \log(x)\right), & \text{if } a > 0, \, a + x \leq 1, \, x > 1, \\
-2  \left( a  \log(a)\right), & \text{if } a > 0, \, a + x \leq 1, \, x \leq 1, \\
(a + x)  \log(a + x) - 2  \left( x  \log(x)\right), & \text{if } a < 0, \, a + x > 1, \, x > 1, \\
(a + x)  \log(a + x) , & \text{if } a < 0, \, a + x > 1, \, x \leq 1, \\
-2  \left( x  \log(x)\right), & \text{if } a < 0, \, a + x \leq 1, \, x > 1, \\
0, & \text{if } a < 0, \, a + x \leq 1, \, x \leq 1.
\end{cases}
\]

Then for $a > 0, \, a + x > 1, \, x > 1$, consider\[
g(x) = (a + x) \log (a + x) - 2 \left(  a \log a + x \log x \right).
\]

For \( x = y - a \), we define:

\[
h(y) = g(x) = y \log y - 2 \left(  a \log a + (-a + y) \log (-a + y) \right).
\]

The derivative of \( h(y) \) is:

\[
h'(y) = -2 (\log a - \log (-a + y)),
\]

and setting \( h'(y) = 0 \) gives the condition \( a = y/2 \).

Substituting this value into \( h(y) \), we obtain:

\[
h(y) =  y \log 4 - y \log y.
\] Then the maximum of this expression is for $y=4/e$. Then $h(4/e)=4/e$ and this is maximum of $g$. Similarly we consider the other cases.
\end{proof}

To prove Theorem~\ref{zyg}, we need the following lemma
\begin{lemma}\label{larmi}
If $f=u+i v=g+\bar h$ is quasiregular, then  $$\Delta |f|\le K^2\Delta (u\log u)$$ provided that $u$ is positive.
\end{lemma}
\begin{proof}[Proof of Lemma~\ref{larmi}] Since $u=\mathrm{Re}\,(g+h)$, it follows that $|\nabla u|=|g'+h'|$. Further

$$|f|= ((g+\bar h)(\bar g +h))^{1/2}.$$
Thus
$$|f|_z= \frac{1}{2}((g+\bar h)(\bar g +h))^{1/2-1} (g'(\bar g +h)+(g+\bar h)h')$$
and
$$|f|_{z \bar z}=-\frac{1}{4}((g+\bar h)(\bar g +h))^{1/2-2} |g'(\bar g +h)+(g+\bar h)h'|^2$$ $$+\frac{1}{2}((g+\bar h)(\bar g +h))^{1/2-1}(|g'|^2+|h'|^2).$$
Then, after some simple manipulations we get  $$\Delta |f|=4|f|_{z \bar z} =\frac{|g'-h'{ f}/{\bar{f}} |^2}{|f|}.$$
Further $$\Delta (u\log u)=\frac{|\nabla u|^2}{u}=\frac{|g'+h'|^2}{u}.$$
Then
$$\frac{\Delta |f|}{\Delta (u \log u)}=\frac{|g'-\frac{ f^2}{|f|^2} h'|^2}{|g'+h'|^2}\frac{u}{|f|}\le \frac{(|g'|+|h'|)^2}{(|g'|-|h'|)^2}\frac{u}{|f|}\le K^2.$$
\end{proof}
\begin{proof}[Proof of Theorem~\ref{zyg}] By \eqref{green}, we have $$|f(0)|=\int_{\mathbb{T}}|f(rz)|\frac{|dz|}{2\pi}-\frac{1}{2\pi}\int_{|z|<r}\Delta |f(z)| \log \frac{r}{|z|}dxdy.$$
Let $\mathbb{D}_r=\{z: |z|<r\}$.
Now,  because $|f(0)|=u(0)$, and using again Proposition~\ref{green},
\[\begin{split}
\int_{\mathbb{T}}|f(rz)|\frac{|dz|}{2\pi}&=u(0)+\frac{1}{2\pi}\int_{\mathbb{D}_r}\Delta |f(z)| \log \frac{r}{|z|}dxdy
\\&\le u(0)+ K^2\frac{1}{2\pi} \int_{\mathbb{D}_r} \Delta (u(z) \log u(z))\log\frac{r}{|z|} dxdy\\&=
 u(0)+K^2\left(-u(0)\log u(0) +\frac{1}{2\pi} \int_{\mathbb{T}} u(rz) \log u(rz)|dz|\right)\\&=\Phi(u(0))+
K^2\frac{1}{2\pi} \int_{\mathbb{T}} u(rz) \log u(rz)|dz|,\end{split}\] where $$\Phi(\xi):= \xi -K^2 \xi \log \xi.$$ Then $
\Phi$ attains its maximum for  $\xi =u(0)=e^{-1+1/K^2}$, which is equal to $ K^2 e^{-1+1/K^2}.$ This implies \eqref{zygmund}.
\end{proof}

\begin{proof}[Proof of Theorem~\ref{zyg44}]
For any fixed $y\neq 0$ we assume that $\{U_{1},\ldots,U_{n-1},N\}$
is a system of mutually orthogonal vectors of the unit norm, where $N=\frac{y}{|y|}$ and
the vectors $\{U_{1},\ldots,U_{n-1}\}$ are arbitrarily chosen.
For a differentiable mapping $h$ defined in domain $\mathbb{B}$
define
 $h_{N}(x)=Dh(x)N$ and set $h_{U_{i}}(x)=Dh(x)U_{i}$ for $i=1,\ldots,n-1$.
Since the Hilbert-Schmidt norm of $Dh$ is independent of basis,
we have
\begin{equation}\label{eq-2.1}
\|Dh\|^{2}
=|h_{N}|^{2}+|h_{U_{1}}|^{2}+|h_{U_{2}}|^{2}+\cdots+|h_{U_{n-1}}|^{2}.
\end{equation}
Let $S(x)=\frac{f(x)}{|f(x)|}$ where $f(x)\neq 0$ because $f_1(x)>0$. Fix $x$ and let  $y=f(x)$ and $N=y/\|y\|$. Then
\begin{equation}\label{eq-2.6}\begin{split} \|DS(x)\|^2
&=|S_{N}(x)|^{2}+|S_{U_{1}}(x)|^{2}+ |S_{U_{2}}(x)|^{2}+\cdots+|S_{U_{n-1}}(x)|^{2}
\\&=|S_{U_{1}}(x)|^{2}+ |S_{U_{2}}(x)|^{2}+\cdots+|S_{U_{n-1}}(x)|^{2},
\end{split}
\end{equation}
because $S_N(x)=0.$ Namely   \[
\sum_{k=1}^{n} \partial_l S_k (x) \cdot N_k (x) = \frac{1}{2} \sum_{k=1}^{n} \partial_l (S_k^2 (x)) = \frac{1}{2} \partial_l \sum_{k=1}^{n} (S_k (x))^2 = \frac{1}{2} \partial_l |S(x)|^2 = 0, \quad l = 1,2, \dots, n.
\]

Since $$S_{U_j}=|f|^{-1} f_{U_j}- |f|^{-3}  \left<f,f_{U_j}\right>f,$$ it follows that $$|S_{U_j}|^2=\frac{|f_{U_j}|^2}{|f|^2}-\frac{\left<f,f_{U_j}\right>^2}{|f|^4}\le \frac{|f_{U_j}|^2}{|f|^2}.$$
Further by straightforward computation (see e.g. \cite[Lemma~1.4]{jmaa1}), $$\Delta |f(x)| =|f(x)|  \|DS(x)\|^2$$ and $$\Delta (f_1\log f_1)=\frac{|\nabla f_1|^2}{f_1}.$$
Since $F$ is $K-$quasiregular,  we have $$|f_{U_j}|\le K|\nabla f_1|,$$ for every $j$.
Thus
\begin{equation}\label{ff1}\frac{\Delta |f|}{\Delta (f_1\log f_1)} \le K^2 (n-1).\end{equation}
\bigskip
Then by \eqref{green} we have
$$\|f\|_1 =|f(0)|+c_n\int_{\mathbb{B}}\Delta|f(x)|\left(|x|^{2-n}-1\right)dV(x),$$ and for $u=f_1$

$$\int_{\mathbb{S}}u\log u d\sigma(\zeta)  =u(0)\log u (0)+c_n\int_{\mathbb{B}}\Delta(u\log u )\left(|x|^{2-n}-1\right)dV(x).$$
Further, we have  $$\|f\|_1-|f(0)|\le K^2(n-1)\left(\int_{\mathbb{S}}\left(u(\zeta)\log u(\zeta)-u(0)\log u (0)\right) d\sigma(\zeta)\right).$$
Then as in proof of Theorem~\ref{zyg}, if $|f(0)|=u(0)$, we have $$\|f\|_1\le K^2(n-1)( e^{-1+K^{-2}/(n-1)}+\|u\log^+u \|_1).$$
Prove now the sharpness part.


Let $f(x) = ax + (1,0,\dots,0)$, $0<a<1$, and define $$X(f)=\|f\|_1-|f(0)|,$$ and $$Y(f)=\left(\int_{\mathbb{S}}\left(u(\zeta)\log u(\zeta)-u(0)\log u (0)\right) d\sigma(\zeta)\right).$$

Then for  $C_n=\frac{\Gamma\left[\frac{n}{2}\right]}{\sqrt{\pi } \Gamma\left[\frac{1}{2} (-1+n)\right]}$ we have $$X(f)=C_n \int_0^{\pi} \sin^{n-2}t  \left(\sqrt{1+a^2+2 a \cos t}-1\right)dt$$

and $$Y(f)=C_n \int_0^{\pi} \sin^{n-2}t (1+a\cos t)\log(1+a\cos t)dt.$$

Since $$\left(\sqrt{1+a^2+2 a \cos t}-1\right)=a\cos(t) +\frac{ a^2}{2} \sin^2 t+o(a^2)$$

and $$(1+a\cos t)\log(1+a\cos t)=a\cos t +\frac{ a^2}{2} \cos^2 t + o(a^2),$$ we get

$$X(f)= C_n\frac{a^2 \sqrt{\pi } \Gamma\left[\frac{1+n}{2}\right]}{2 \Gamma\left[1+\frac{n}{2}\right]}+o(a^2)=\frac{a^2 (n-1)}{2 n}+o(a^2)$$

$$ Y(f)= C_n \frac{a^2 \sqrt{\pi } \Gamma\left[\frac{1}{2} (-1+n)\right]}{4 \Gamma\left[1+\frac{n}{2}\right]}+o(a^2)=\frac{a^2 }{2 n}+o(a^2).$$
Then we obtain  $$\lim_{a\to 0}\frac{X(f)}{Y(f)}=n-1.$$ Hence the constant $n-1$ is the best possible for the case $K=1$.
\end{proof}
\subsection*{Funding} This research received no specific grant from any funding agency in the public, commercial, or not-for-profit sectors.
\subsection*{Data availibility} Data sharing is not applicable to this article since no data sets were generated or analyzed
\section*{Ethics declarations}
The author declares that he has not conflict of interest.

\subsection*{Acknowledgment}
 I sincerely thank the referee for their valuable comments and insightful suggestions, which have significantly improved the quality and clarity of this manuscript.

\normalsize

\end{document}